\documentclass[12pt]{amsart}
\usepackage[hyperfootnotes=true]{hyperref}

\usepackage{amsfonts,amssymb}
\usepackage{dsfont}
\usepackage{amscd}
\usepackage{extarrows}
\usepackage{amsmath}
\usepackage{enumerate}
\usepackage{amscd}
\usepackage[all]{xy}
\usepackage[hyperfootnotes=true]{hyperref}
\usepackage{mathrsfs}

\theoremstyle{fancy}
\newtheorem{theorem}{Theorem}[section]

\newtheorem{lemma}[theorem]{Lemma}

\newtheorem{remark}[theorem]{Remark}
\newtheorem{question}[theorem]{Question}
\newtheorem{example}[theorem]{Example}

\newcommand{\cL}{\mathcal{L}}

\begin{document}
\title{Local Demailly-Bouche's holomorphic Morse inequalities}
\author{Zhiwei Wang}
\address{ School
of Mathematical Sciences\\Beijing Normal University\\ Beijing 100875\\ P. R. China}
\email{zhiwei@bnu.edu.cn}

\thanks{The  author was partially supported by the Fundamental Research Funds for the Central Universities and by the NSFC grant NSFC-11701031}

\begin{abstract}Let $(X,\omega)$ be a Hermitian manifold and let $(E,h^E)$, $(F,h^F)$ be  two  Hermitian holomorphic line bundle over $X$. Suppose  that the maximal rank of the Chern curvature $c(E)$ of $E$ is $r$, and the kernel of $c(E)$ is foliated, i.e. there is a foliation $Y$ of $X$, of complex codimension $r$, such that the tangent space of the leaf at each point $x\in X$ is contained in the kernel of $c(E)$.  In this paper,   local versions of Demailly-Bouche's holomorphic Morse inequalities (which give asymptotic  bounds for cohomology groups  $H^{q}(X,E^k\otimes F^l)$ as $k,l,k/l\rightarrow \infty$) are presented.  The local version holds on any Hermitian manifold regardless of compactness and completeness. The proof is a variation of Berman's method to  derive  holomorphic Morse inequalities on compact complex manifolds with boundary.
\end{abstract}
\subjclass[2010]{32A25, 32L10, 32L20}
\keywords{Bergman kernel, Holomorphic Morse inequalities, Localization}

\maketitle

\setlength{\baselineskip}{17pt}

\section{Introduction}

Let $X$ be a compact complex manifold of complex dimension $n$. Let $(L,h^L)$  (resp. $(E,h^E)$) be a hermitian holomorphic line (resp. vector ) bundle over $X$, where the rank of $E$ is $r$. Let $\nabla^L$ be Chern connection of $L$ with respect to the hermitian metric $h^L$, and $c(L)=\frac{\sqrt{-1}}{2\pi}(\nabla^L)^2$ be the Chern curvature. Denote by $X(q)$ the open subset of $X$ on which  $c(L)$ has strictly $q$ negative eigenvalues and $n-q$ positive eigenvalues and $X(\leq q)=\cup_{i\leq q}X(i)$. Let $H^q(X,L^k\otimes E)$ be the cohomology group which is isomorphic to $H^{0,q}_{\overline{\partial}}(X,L^k\otimes E)$ by the Dolbeault lemma \cite{GH}.

The celebrated    holomorphic Morse inequalities read that as $k\rightarrow \infty$, the following estimates hold
\begin{align}
\dim_\mathbb{C} H^q(X, L^k\otimes E)\leq (-1)^q r\frac{k^n}{n!}\int_{X(q)}(c(E))^n+o(k^n),\label{Dem1} \\
\sum_{0\leq j\leq q}(-1)^{q-j}\dim_\mathbb{C} H^j(X,L^k\otimes E)\leq (-1)^q r\frac{k^n}{n!}\int_{X(\leq q)}(c(E))^n+o(k^n)\label{Dem2}.
\end{align}
Where (\ref{Dem1}) and (\ref{Dem2}) are called weak and strong Holomorphic Morse inequalities respectively.

 Holomorphic Morse inequalities were first introduced in \cite{Dem 85} by Demailly to improve Siu's solution of the Grauert-Riemenschneider conjecture in \cite{Siu 84} which states that if $X$ carries  line bundle $L$ with a smooth Hermitian metric $h$, such that the associated Chern curvature $i\Theta_h$  is semi-positive， and positive at least at  one point,  then $X$ is Moishezon, i.e. birational to a projective variety. It is worth to mention that, Demailly's holomorphic Morse inequalities give a criterion of a line bundle to be big  only  through the positivity of a integration of the  curvature form.  It is proved by Ji-Shiffman \cite{JiS93} that, a compact complex manifold is Moishezon, if and only if it carries a holomorphic line bundle equppied with a singular Hermitian metric such that the curvature current associated to the singular metric is a K\"{a}hler current.  A closed real $(1,1)$-current $T$ is said to be a K\"{a}hler current, if there is a $\varepsilon>0$, such that $T\geq \varepsilon\omega$ in the sense of current. Namely, $(L,h)$ is big if and only if  $\int_{X(\leq 1)}(i\Theta_h)^n>0$.  Moreover, holomorphic Morse inequalities express asymptotic bounds on the individual cohomology of tensor bundles of holomorphic line bundles, it is a useful complement to the Riemann-Roch formula. Holomorphic Morse inequalities are  applied to many situations to study important problems \cite{Berman 06, Bou02,Cao13, Dem11,Pop08,Pop13,Ya17,Ye09,W16}.

 Recently, there are may generalizations of Demailly's holomorphic Morse inequalities   \cite{Berman 04, Berman 05, Bon98, Bouche 91, Hs14, HL15, HL16,HM12, MM07, Ma92,Ma96,Pu17, TCM01}  due to increasing  interests in different research topics. One of highlights, is the recent work of Hsiao-Li \cite{HL15, HL16},  they derived    CR versions of holomorphic Morse inequalities, which closely relates to the embedding problem of CR manifold with transversal CR $S^1$-action. We refer to \cite{HL15, HL16} and reference therein.

  The original proof of  Demailly was  inspired by Witten's analytic proof of the classical Morse inequalities for the Betti numbers of a compact real manifold \cite{Witten 82}.  Subsequently Bismut \cite{Bis 87},  Bouche \cite{Bouche 94}, Demailly \cite{Dem89}, Ma-Marinescu \cite{MM07} gave alternative proofs through asymptotic estimates of the heat kernel of the $\overline{\partial}$-Laplacian which were quite delicate analytic arguments.   Recently, Berndtsson \cite{Ber 03},  Berman \cite{Berman 04} indicated the  proof of  these inequalities in an  elementary way based on the estimate for the Bergman kernel for the space $H^q(X,L^k\otimes E)$. In fact, Berman \cite{Berman 04} proved a local version of holomorphic Morse inequalities regardless whether the manifold is compact or non-compact, which could be used  to study other problems in complex geometry, such as asymptotics of eigenvalues of  super Toeplitz operator, sampling sequences \cite{Berman 06}.

From Demailly's holomorphic  Morse inequalities, we know that if $c(L)$ is not everywhere degenerate, then there always exist $q$ such that $X(q)\neq \emptyset$, thus $h^q\sim Ck^n$. But if  $c(L)$ is degenerate everywhere, the only thing we know from Demailly's holomorphic Morse inequalities is that $h^q\sim o(k^n)$. It is natural to ask can we get a better estimate of the term $o(k^n)$?  It was not until Bouche \cite{Bouche 91},  who proved the following theorem, which  can make us  understand the term $o(k^n)$ in the degenerate case better.

\begin{theorem}[c.f.\cite{Bouche 91}]\label{Bou} Let $(X,\omega)$ be a compact connected Hermitian manifold of complex dimension $n$.
Suppose $(E,h^E)$, $(F,h^F)$ are two hermitian holomorphic line bundle over $X$, $(G,h^G)$ be a hermitian holomorphic vector bundle of rank $g$ over $X$. Let $\nabla^ E$, $\nabla^F$ be the Chern connection of $E$, and $F$, and  $c(E)=\frac{\sqrt{-1}}{2\pi}(\nabla^E)^2$, $c(F)=\frac{\sqrt{-1}}{2\pi}(\nabla^F)^2$ be the Chern curvature form of $E$ and $F$ respectively. Assume that the maximal rank of $ic(E)$ is $r$ and the kernel of $c(E)$ is foliated, i.e. there is a foliation $Y$ of $X$, of codimension $r$, such that the tangent space of the leaf at each point $x\in X$ is contained in the kernel of $c(E)_x$.  This
is automatically satisfied if the rank of the kernel of the form $c(E)$ is constant since $c(E)$ is a closed form. For any direct sum decomposition of $TX=TY\oplus NY$, we can define a real
$(1,1)$-form $\Theta$ which equals to the direct sum of the form induced by $c(E)$ on $NY$ and by $c(F)$ on $TY$. For $q=0,1,\cdots, n,$ we define $X(q)$, the open set on which $\Theta$ have exactly $q$ negative eigenvalues and $n-q$ positive eigenvalues, and $X(\leq q)=\cup_{i\leq q}X(i)$. Then by letting $k\rightarrow +\infty$, $l\rightarrow +\infty$, and $\frac{k}{l}\rightarrow +\infty$, we have the following inequalities
\begin{align}
 \dim_\mathbb{C} H^q(X, E^k\otimes F^l\otimes &G )\label{Bou1}\\
\leq (-1)^q g\frac{k^r}{r!}\frac{l^{n-r}}{(n-r)!}&\int_{X(q)}(c(E))^r\wedge(c(F))^{n-r}+o(k^rl^{n-r}),\notag \\
\sum_{0\leq j\leq q}(-1)^{q-j}\dim_\mathbb{C} H^j(X,&E^k\otimes F^l\otimes G)\label{Bou2}\\
\leq (-1)^q g\frac{k^r}{r!}&\frac{l^{n-r}}{(n-r)!}\int_{X(\leq q)}(c(E))^r\wedge(c(F))^{n-r}+o(k^rl^{n-r}).\notag
\end{align}
  \end{theorem}

It  is pointed out by Bouche \cite{Bouche 91} that,  from the above Theorem, one can  get that if $E$ satisfies the same condition as in the above theorem, then when $k$ tends to $+\infty$, we have $\dim_\mathbb{C} H^q(X, E^k\otimes G)\leq Ck^r$.  This makes the term $o(k^n)$ more like "it should be".

The original proof in \cite{Bouche 91} of Theorem \ref{Bou} was along the way of Demailly \cite{Dem 85}, but doing more delicate estimates.  It is natural to ask if  we can get a local version of Demailly-Bouche's  holomorphic Morse inequalities and then we can find more applications, e.g,  in asymptotics of  super Toeplitz operator and sampling sequences \cite{Berman 06}. This is the main motivation of this paper. 

Let $(X, \omega)$ be a compact complex manifold with complex dimension  $n$,   $(E, h^E)$ and $(F,h^F)$ be holomorphic hermitian line bundles over $X$.
 By Hodge theory,  $H^{0,q}_{\overline{\partial}}(X,E^k\otimes F^l)$, the $(0,q)$-th Dolbeault
cohomology group is of finite dimension and
isomorphic to the space of harmonic forms with values in $E^k\otimes
F^l$ which was denoted by $\mathcal{H}^{0,q}(X,E^k\otimes
F^l)$. There is a canonical $L^2$ metric of
this space induced from the hermitian metric of $E,F$ and of the
hermitian metric of the manifold $X$. So we can choose an
orthonormal basis $\{\Psi_1,\cdots, \Psi_N\}$ of
$\mathcal{H}^{0,q}(X,E^k\otimes F^l)$, then we define the
so-called Bergman kernel function and the extremal function to be
\begin{align}
B^{q,k,l}_X(x)=\sum_{i=1}^N|\Psi_i(x)|^2.\notag
\end{align}
\begin{align}
S^{q,k,l}_X(x)=\sup\frac{|\alpha(x)|^2}{\| \alpha \|^2_X}.\notag
\end{align}

There is  also  a component version  of $S^{q,k,l}_X(x)$. For a given
orthonormal frame $e^I_x$ in $\wedge^{0,q}_x(X,E^k\otimes F^l)$, set
\begin{align}
S^{q,k,l}_{X,I}(x)=\sup\frac{|\alpha_I(x)|^2}{\| \alpha \|^2_X}.\notag
\end{align}

It is proved in \cite{Berman 04}  that the above kernels satisfy the following inequality
\begin{align}\label{lemma 2.1}
S^{q,k,l}_X(x)\leq B^{q,k,l}_X(x)\leq \sum_IS^{q,k,l}_{X,I}(x).
\end{align}

 Now, we are on the way to  introduce the main results in this paper.  Firstly, by adapting Berman's localization technique to our situation, we get the following local version of Demailly-Bouche's weak holomorphic Morse inequalities.
\begin{theorem}\label{Main theorem 1}
Given a (compact or noncompact) Hermitian manifold $(X,\omega)$, let $E, F$ be two holomorphic line bundles satisfying the assumption in Theorem \ref{Bou}. We can define metrics $\omega_{k,l}$ and $\omega_0$ (see Section \ref{localization}) on $X$. Then the Bergman kernel function $B^{q,k,l}_X$ and the extremal function $S^{q,k,l}_X$ of the space of the global $\overline{\partial}$-harmonic $(0,q)$-forms with values in $E^k\otimes F^l$, satisfy
\begin{align}
\limsup_{k,l;\frac{k}{l}\rightarrow \infty}B^{q,k,l}_X(x)\leq B^q_{x,\mathbb{C}^n}(0),~~~  \limsup_{k,l,\frac{k}{l}\rightarrow \infty}S^{q,k,l}_X(x)\leq S^q_{x,\mathbb{C}^n}(0),\notag
\end{align}
where
\begin{align}
B^q_{x,\mathbb{C}^n}(0)=S^q_{x,\mathbb{C}^n}(0)=1_{X(q)}(x)\big|det_{\omega_0}(\Theta)_x\big|,\notag
\end{align}
and $\lim_{k,l,\frac{k}{l}}B^{q,k,l}_X(x)=\lim_{k,l,\frac{k}{l}}S^{q,k,l}_X(x)$ if one of the limits exists.
\end{theorem}

Furthermore, when $X$ is compact, we show that  the above local Demailly-Bouche's weak holomorphic Morse inequalities can be extended to an asymptotic equality. Denote by $ B^{q,k,l}_{\leq \mu_{k,l}}$   the Bergman kernel function of the space spanned by all the eigenforms of the $\overline{\partial}$-Laplacian, whose eigenvalues are bounded by $\mu_{k,l}$.
\begin{theorem}\label{Main theorem 2}
Under the same assumption as in Theorem \ref{Main theorem 1}, one has that
\begin{align}
\lim\limits_{k,l,\frac{k}{l}\rightarrow \infty} B^{q,k,l}_{\leq \mu_{k,l}}(x)=1_{X(q)}(x)\big|det_{\omega_0}(\Theta)_x\big|.\notag
\end{align}
for some sequence $\mu_{k,l}$ tending to zero.
\end{theorem}

 It is worth  to mention that in \cite{Ber 02}, Berndtsson pointed out that the precise relation of Bouche's results (Theorem \ref{Bou}) and  his eigenvalue estimate is for the moment not clear.  Here Berndtsson's eigenvalue estimate is stated as follows.
 
 Assume $L$ is given a hermitian metric of semipositive curvature. Take $q\geq 1$, then if $0\leq \lambda\leq k$, we have
\begin{align}
h^{n,q}_{\leq \lambda}(X,L^k\otimes E)\leq C(1+\lambda)^qk^{n-q}.\notag
\end{align}
If $1\leq k\leq \lambda$, we have
\begin{align}
h^{n,q}_{\leq \lambda}(X,L^k\otimes E)\leq C\lambda^n,\notag
\end{align}
where $h_{\leq \lambda}(X,L^k\otimes E)$ is the dimension of the linear span of the eigenforms of the $\overline{\partial}$-Laplacian, whose eigenvalue is less than or equal to $\lambda$.

Now let $L$ be a semipositive holomorphic line bundle, and the maximal rank of the  Chern curvature form $c(L)$ is $r$ such that $r\neq 0$, then the  question is that if we can  improve the estimate in the above Berndtsson's result? Observed that the key ingredient in the  proof of Berndtsson's esimate  \cite{Ber 02} is a localization technique,  which is used  by Berman  \cite{Berman 04} to deduce the local holomorphic Morse inequalities.   For this consideration, the local version of Demailly-Bouche's holomorphic Morse inequalities is a good  start to investigate this question.

The structure of this paper is as follows. In Section 2, we introduce the Bergman kernel functions and the extremal functions and give an inequality which relates these two functions. In Section 3, we give the philosophy of localization including some basics of elliptic operator and the model extremal functions and Bergman kernel functions. In Section 4, we give a proof of the weak version of the Demailly-Bouche's holomorphic Morse inequalities. In Section 5, we give a proof of the strong version of the Demailly-Bouche's holomorphic Morse inequalities. We also raise a question on generalizing Demailly-Bouche's holomorphic Morse inequalities to CR setting.

\section{Localization  Procedure}\label{localization}
Throughout this paper, we assume that the condition in Theorem \ref{Bou} holds unless otherwise is specified.

Let $\omega_0$ be a hermitian metric of $X$, at each point $x\in X$, the orthogonal space of $T_xY$ with respect to $\omega_0$ defines a vector bundle of complex rank $n-r$ which was denoted by $NY$. Choose $\eta$ (resp. $\zeta$) a hermitian metric on $NY$ (resp. $TY$). Without loss of generality, we let $\omega_0=\eta+\zeta$. Define a hermitian metric $\omega_{k,l}=k\eta+l\zeta$ on $X$. The volume form on $X$ with respect to the Hermitian metric $\omega_{k,l}$ is
\begin{align*}
\frac{\omega_{k,l}^n}{n!}&=\sum_{i=0}^n\binom{i}{n}(k\eta)^i\wedge(l\zeta)^{n-i}\\
&=\binom{r}{n}(k\eta)^{r}\wedge(l\zeta)^{n-r}\\
&=\frac{k^rl^{n-r}}{r!(n-r)!}\eta^{r}\wedge \zeta^{n-r}\\
&=\frac{k^rl^{n-r}}{n!}\omega_0^n.
 \end{align*}
 Note that the fibers $NY$ and $TY$ are orthogonal with respect to the metric $\omega_{k,l}$.
 Around each point $x\in X$, we can find a local complex coordinate $\{z_1,\cdots, z_n\}$,  such that $Y=\{z_1=\cdots=z_r=0\}$ and
\begin{align}
\eta(z)&=\sqrt{-1}\sum_{i,j=1}^{r} h_{i,j}(z)dz_i\wedge \overline{dz_j}, ~~ h_{i,j}(0)=\delta_{i,j},\notag\\
\zeta(z)&=\sqrt{-1}\sum_{i,j=r+1}^nh_{i,j}(z)dz_i\wedge \overline{dz_j}, ~~ h_{i,j}(0)=\delta_{i,j}.\notag
\end{align}
By choosing suitable local frame of $E$ and $F$, up to orthonormal transformation to $(z_1,\cdots, z_r)$ and $(z_{r+1},\cdots, z_n)$ respectively, we can get that 
\begin{align}
\phi(z)&=\sum_{i=1}^r\lambda_{i,x}|z_i|^2+\sum_{1\leq j\leq n,r+1\leq  i\leq  n}(\lambda_{ij,x}z_i\overline{z}_j+\overline{\lambda_{ij,x}}\overline{z_i}z_j)+O(|z|^3),\notag\\
\psi(z)&=\sum_{i=r+1}^{n}\nu_{i,x}|z_i|^2+\sum_{1\leq j\leq n, 0\leq i\leq r}(\nu_{ij,x}z_i\overline{z}_j+\overline{\nu_{ij,x}}\overline{z_i}z_j)+O(|z|^3).\notag
\end{align}

By our assumption, the tangent space of $Y$ at any point $x\in X$ is contained in the kernel $c(E)$,  we have that $\lambda_{i,j}=0$ for $1\leq j\leq n,r+1\leq  i\leq  n$. It is easy to see that
 \begin{align}
 c(E)&=\frac{\sqrt{-1}}{2\pi}\sum_{i=1}^r\lambda_{x,i}dz_i\wedge\overline{dz_i}+O(|z|)\notag\\
 c(F)&=\frac{\sqrt{-1}}{2\pi}\sum_{i=r+1}^n\nu_{x,i}dz_i\wedge\overline{dz_i}+\mbox{mixed terms}+O(|z|).\notag
 \end{align}
Set
 \begin{align}
\Theta(z)&=c(E)|_{NY}+c(F)|_{TY}\notag\\
 &=\frac{\sqrt{-1}}{2\pi}\sum_{i=1}^r\lambda_{x,i}dz_i\wedge d\overline{z}_i+\frac{\sqrt{-1}}{2\pi}\sum_{i=r+1}^n\nu_{x,i}dz_i\wedge d\overline{z}_i+O(|z|),\notag\\
 \gamma_0(z)&=\sum_{i=1}^r\lambda_{x,i}|z_i|^2+\sum_{i=r+1}^n\nu_{x,i}|z_i|^2.\notag
 \end{align}

Set $r_{k,l}=\log\min\{\frac{k}{l}, l\}$. Then for any sequence of $k,l\in \mathbb{N}$ such that $k\rightarrow +\infty$, $l\rightarrow +\infty$ and $\frac{k}{l}\rightarrow +\infty$, one has  $r_{k,l}\rightarrow +\infty$.

 Let $z'=(z_1,\cdots,z_r)$ and $z''=(z_{r+1},\cdots, z_n)$, then $z=(z',z'')$.  Let $B_{r_{k,l}}:=\{z| |z'|<{r_{k,l}/\sqrt{k}}, |z''|<{r_{k,l}/\sqrt{l}}\}$ be identified with a small open subset in the coordinate chart of $x$.

 Define a scaling mapping $f^{(k,l)}$ from $B_{|z'|<r_{k,l}, |z''|<r_{k,l}}$ to $B_{r_{kl}}$ by
\begin{align}
f^{(k,l)}(z):=f(\frac{z'}{\sqrt{k}},\frac{z''}{\sqrt{l}}).\notag
\end{align}

For any given object $\alpha$ defined on the manifold $X$, denote by $\alpha^{(k,l)}$  the scaling of $\alpha$ restricted to $B_{r_{k,l}}$, i.e.
 \begin{align*}
 \alpha^{(k,l)}=(f^{(k,l)})^*(\alpha).
 \end{align*}

 By direct computation, we have
\begin{align}
(k\phi)^{(k,l)}(z)&=\sum_{i=1}^r\lambda_{x,i}|z_i|^2+kO_1(k,l)(|z|^3),\notag\\
(l\psi)^{(k,l)}(z)&=\sum_{1\leq j\leq n, 0\leq i\leq r}\sqrt{\frac{l}{k}}Re(\nu_{ij,x}z_i\overline{z}_j)+\sum_{i=r+1}^n\nu_{x,i}|z_i|^2+lO_2(k,l)(|z|^3),\notag
\end{align}
where $Re(\nu_{ij,x}z_i\overline{z}_j)$ is the real part of $\nu_{ij,x}z_i\overline{z}_j$.

 For later use, we analyse the two terms $O_1(k,l)$ and $O_2(k,l)$ a little bit more. Since $c(E)$ contains $TY$ as its kernel, by computing $\partial\overline{\partial}\phi$, we know that the term $O(|z|^3)$ in the expansion of $\phi$ should contain at least two coordinate functions from $\{z_i\}_{1\leq i\leq r}$. This implies that  $O_1(k,l)\leq \frac{1}{k\sqrt{ l}}$. In general $O_2(k,l)\leq \frac{1}{(\sqrt{l})^3}$.  From the choice of $r_{k,l}$, easy computations show that
\begin{align}
&\sup_{|z'|<r_{k,l}, |z''|<r_{k,l}}\big|\partial^{\alpha}((k\phi)^{(k,l)}+(l\psi)^{(k,l)}-\gamma_0)(z)\big|\rightarrow 0,\label{metric 1}\\
&\sup_{|z'|<r_{k,l}, |z''|<r_{k,l}}\big|\partial^{\alpha}(\omega_{k,l}^{(k,l)}-\sum_{i=1}^ndz_i\wedge d\overline{z}_i)\big|\rightarrow 0,\label{metric 2}
\end{align}
when $k,l\rightarrow +\infty$ and $\frac{k}{l}\rightarrow +\infty$.

Moreover, $B_{|z'|<r_{k,l}, |z''|<r_{k,l}}$ exhausts $\mathbb{C}^n$, when $k,l\rightarrow +\infty$.

We also have the following fact
\begin{align}\label{Norm invariance}
f^*_{k,l}|\alpha_{k,l}|^2=|\alpha^{(k,l)}|^2
\end{align}
where the first norm is taken with respect to the metric $\omega_{k,l}$ and the fiber metric $k\phi$ and $l\psi$, and the second norm is taken with respect to the scaled metric  $\omega_{k,l}^{(k,l)}$ and the scaled fiber metric $(k\phi)^{(k,l)}$ and $(l\psi)^{(k,l)}$.

Denote by $\Delta_{k,l}$ the $\overline{\partial}$-Laplacian defined with respect to the metric $\omega_{k,l}$ and the fiber metric $k\phi$ and $l\psi$, by $\Delta_{\overline{\partial}}^{(k,l)}$ the $\overline{\partial}$-Laplacian defined  with respect to the scaled metric $\omega_{k,l}^{(k,l)}$ and  the scaled fiber metric $(k\phi)^{(k,l)}$ and $(l\psi)^{(k,l)}$.

 As stated in \cite[Lemma 5.2]{Berman 05}, the Laplacian is naturally defined with respect to any given metric, it is invariant under pull-back, we thus have the following identity
\begin{align}
\Delta_{\overline{\partial}}^{(k,l)}\alpha^{(k,l)}=(\Delta_{k,l}\alpha)^{(k,l)}.\notag
\end{align}

Moreover, from (\ref{metric 1}), (\ref{metric 2}) and similar argument  as \cite[(5.3)]{Berman 05}, one can obtain that 
\begin{align}\label{Laplacian expansion}
\Delta_{\overline{\partial}}^{(k,l)}=\Delta_{\overline{\partial},\gamma_0}+\epsilon_{k,l}\mathcal{D}_{k,l},
\end{align}
where $\Delta_{\overline{\partial},\gamma_0}$ is the $\overline{\partial}$-Laplacian defined with respect to the Euclidean metric on $\mathbb{C}^n$ and the metric $e^{-\gamma_0}$ of the trivial line bundle over $\mathbb{C}^n$, and $\mathcal{D}_{k,l}$ is a second order partial differential operator with bounded variable coefficients on the scaled ball $B_{|z'|<r_{k,l}, |z''|<r_{k,l}}$ and  $\epsilon_{k,l}$ is a sequence tending to zero with $k, l, \frac{k}{l}$ tending  to $+\infty$.

Observed  that for any form $\alpha$ with values in $E^k\otimes F^l$, one can get that 
\begin{align}\label{norm localization}
\|\alpha\|^2_{B_{r_{k,l}}}\sim \|\alpha^{(k,l)}\|^2_{\gamma_0,|z'|<r_{k,l}, |z''|<r_{k,l}}.
\end{align}

Form standard techniques for elliptic operators,  and similarly with \cite[Lemma 3.1]{Berman 04}, the following lemma holds.
\begin{lemma}[{c.f. \cite[Lemma 3.1]{Berman 04}}]\label{lemma 3.1}
For each $k, l$, suppose that $\beta^{(k,l)}$ is a smooth form on $B_{|z'|<r_{k,l}, |z''|<r_{k,l}}$ such that $\Delta^{(k,l)}_{\overline{\partial}}\beta^{(k,l)}=0$. Identify $\beta^{(k,l)}$ with a form in $L^2_{\gamma_0}(\mathbb{C}^n)$ by extending with zero. Then there is constant $C$ independent of $k,l$ such that
\begin{align}
\sup_{z\in B_1}|\beta^{(k,l)}(z)|_{\gamma_0}^2\leq C\|\beta^{(k,l)}\|^2_{\gamma_0,B_2}.\notag
\end{align}
Moreover, if the sequence of norms $\|\alpha^{(k,l)}\|_{\gamma_0,C^n}^2$ is bounded, then there is a subsequence of $\beta^{(k,l)}$  which converges uniformly with all derivatives on any ball in $\mathbb{C}^n$ to a smooth form $\beta$, where $\beta$ is in $L^2_{\gamma_0}(\mathbb{C}^n)$.
\end{lemma}

Consider a model case, i.e. a trivial line bundle over $\mathbb{C}^n$ equipped with Hermitian metric $\gamma_0$. Simple computations from calculus imply that  
 \begin{align}\label{model kernel}
 B^q_{x,\mathbb{C}^n}(0)=S^q_{x,\mathbb{C}^n}(0)=1_{X(q)}(x)\big|det_{\omega_0}(\Theta)_x\big|.
 \end{align}
Moreover, suppose that the first $q$ eigenvalues of the quadratic form $\gamma_0$ are negative and the rest are positive (which corresponds to the case when $x$ is in the open subset $X(q)$). Then
\begin{align}\label{extremal function}
S^q_{I,x,\mathbb{C}^n}(0)=0,
\end{align}
unless $I=(1,2,\cdots,q)$.  For the detailed  proof, we refer to \cite[Proposition 4.3]{Berman 04}.
\section{Local version of Demailly-Bouche's weak holomorphic Morse inequalities}
In this section, we are going to prove the following local version of Demailly-Bouche's weak holomorphic Morse inequalities.

\begin{theorem}[=Theorem \ref{Main theorem 1}]
Let $(X,\omega_{k,l})$ be a hermitian manifold. Let $E, F$ satisfies the assumption in Theorem \ref{Bou}. Then the Bergman kernel function $B^{q,k,l}_X$ and the extremal function $S^{q,k,l}_X$ of the space of the global $\overline{\partial}$-harmonic $(0,q)$-forms with values in $E^k\otimes F^l$, satisfy
\begin{align}
\limsup_{k,l;\frac{k}{l}\rightarrow \infty}B^{q,k,l}_X(x)\leq B^q_{x,\mathbb{C}^n}(0),~~~  \limsup_{k,l,\frac{k}{l}\rightarrow \infty}S^{q,k,l}_X(x)\leq S^q_{x,\mathbb{C}^n}(0),\notag
\end{align}
where
\begin{align}
B^q_{x,\mathbb{C}^n}(0)=S^q_{x,\mathbb{C}^n}(0)=1_{X(q)}(x)\big|det_{\omega_0}(\Theta)_x\big|,\notag
\end{align}
and $\lim\limits_{k,l,\frac{k}{l}\rightarrow \infty}B^{q,k,l}_X(x)=\lim\limits_{k,l,\frac{k}{l}\rightarrow\infty}S^{q,k,l}_X(x)$ if one of the limits exists.
\end{theorem}
\begin{proof}
Firstly, we will prove that
\begin{align}
\limsup_{k,l,\frac{k}{l}\rightarrow \infty}S^{q,k,l}_X(x)\leq S^q_{x,\mathbb{C}^n}(0).\notag
\end{align}

By definition, there is a  sequenence $\alpha_{k,l}\in H^{0,q}(X,E^k\otimes F^l)$, such that
\begin{align}
\|\alpha_{k,l}\|&=1, \notag\\
\limsup_{k,l,\frac{k}{l}\rightarrow \infty}S^{q,k,l}_X(x)&=\limsup_{k,l,\frac{k}{l}\rightarrow \infty}|\alpha_{k,l}(x)|^2.\notag
\end{align}

Now consider the sequence $\beta^{(k,l)}$, which equals to $\alpha^{(k,l)}$ on $B_{|z'|<r_{k,l}, |z''|<r_{k,l}}$ and   identified with a form in $L^2_{\gamma_0}(\mathbb{C}^n)$, by extending with zero.

Note that
\begin{align*}
\limsup_{k,l,\frac{k}{l}\rightarrow\infty}\|\beta^{(k,l)}\|_{\gamma_0}^2&=\limsup_{k,l,\frac{k}{l}\rightarrow \infty}\|\alpha^{(k,l)}\|^2_{\gamma_0,B_{|z'|<r_{k,l}, |z''|<r_{k,l}}}\\
&\sim \limsup_{k,l,\frac{k}{l}\rightarrow \infty}\|\alpha_{(k,l)}\|^2_{B_{r_{k,l}}}\\
&\leq \limsup_{k,l,\frac{k}{l}\rightarrow \infty}\|\alpha_{k,l}\|_X^2=1,
\end{align*}
where the second estimate follows from  (\ref{norm localization}). 

From  Lemma \ref{lemma 3.1}, there is a subsequence $\beta^{(k_j,l_j)}$ that converges uniformly with all derivatives to $\beta$ on any ball in $\mathbb{C}^n$, where $\beta$ is smooth and $\|\beta\|^2_{\gamma_0,\mathbb{C}^n}\leq 1$. Hence we have $\Delta_{\overline{\partial},\gamma_0}\beta=0$, which follows from (\ref{Laplacian expansion}), implying that 
\begin{align}
\limsup_{k,l,\frac{k}{l}\rightarrow \infty}S^{q,k,l}_X(x)=\lim_{j\rightarrow \infty}|\beta^{(k_j,l_j)}(0)|^2=|\beta(0)|^2\leq \frac{|\beta(0)|^2}{\|\beta\|^2_{\gamma_0,\mathbb{C}^n}}\leq S^{q}_{x,\mathbb{C}^n}(0),\notag
\end{align}
where the first equality follows from (\ref{Norm invariance}).

Moreover, from
\begin{align}
B^q_{x,\mathbb{C}^n}(0)=S^q_{x,\mathbb{C}^n}(0)=1_{X(q)}(x)\big|det_{\omega_0}(\Theta)_x\big|,\notag
\end{align}
and Lemma \ref{lemma 2.1}, we can see that $\lim\limits_{k,l,\frac{k}{l}\rightarrow \infty}B^{q,k,l}_X(x)=0$ outside $X(q)$.

Next if $x\in X(q)$, we may assume $\lambda_1,\cdots, \lambda_q$ are the negative eigenvalues. By (\ref{extremal function}), we have that $\beta^I=0$, if $I\neq (1,\cdots, q)$. We obtain that if $I\neq (1,\cdots, q)$,
\begin{align}
\lim_{k_j,l_j,\frac{k_j}{l_j}\rightarrow \infty}S^{q,k_j,l_j}_I(0)=\lim_{k_j,l_j,\frac{k_j}{l_j}\rightarrow \infty}|\alpha^I_{k_j,l_j}(0)|^2=|\beta^I(0)|^2=0.\notag
\end{align}

This  proves that
\begin{align}
\lim_{k,l,\frac{k}{l}\rightarrow \infty}  S^{q,k,l}_I(0)=0\notag
\end{align}
if $I\neq (1,\cdots, q)$. 

Finally, from Lemma \ref{lemma 2.1}, we deduce that
\begin{align}
\lim_{k,l,\frac{k}{l}\rightarrow \infty}B^{q,k,l}_X(x)\leq 0+\cdots+0+S^q_{x,\mathbb{C}^n}(0)=B^q_{x,\mathbb{C}^n}(0).\notag
\end{align}
 The proof of this theorem is thus completed.
\end{proof}
\section{The Weak Version of Demailly-Bouche's Holomorphic Morse Inequalities}
In this section, by using local version of Demailly-Bouche's weak holomorphic Morse inequalities, we give a simple proof of  Demailly-Bouche's weak holomorphic Morse inequalities.
\begin{theorem}[c.f.\cite{Bouche 91}]\label{weak holomorphic morse inequalities}
Suppose $X$ is compact. $E$, $F$ are two holomorphic line bundles which satisfy the assumption in Theorem \ref{Bou}, then for any $q=0,1,\cdots, n$, when $k,l,\frac{k}{l}$ tends to infinity, we have
\begin{align}
&\dim_\mathbb{C} H^q(X,E^k\otimes F^l)\notag\\
&\leq \frac{k^r}{r!}\frac{l^{n-r}}{(n-r)!}\int_{X(q)}(-1)^q\left(\frac{i}{2\pi}c(E)\right)^r\wedge \left(\frac{i}{2\pi}c(F)\right)^{n-r}+ o(k^rl^{n-r}).\notag
\end{align}
\end{theorem}
\begin{proof}
We first show that the sequence $S^{q,k,l}_X(x)$ is dominated by a constant if $X$ is compact. 

Since $X$ is compact, it is sufficient to prove this for a sufficiently small neighborhood of a fixed point $x_0$. 

For a given form $\alpha_{k,l}\in \mathcal{H}^{0,q}(X,E^k\otimes F^l)$, we consider its restriction to a polydisk $B_{\frac{1}{\sqrt{k}},\frac{1}{\sqrt{l}}}$ centered at $x_0$.

Using G{\aa}rding's inequality, we see that there is a constant $C(x_0)$ depending continuously on $x_0$, such that
\begin{align}
|\alpha_{k,l}(x_0)|^2\sim|\alpha^{(k,l)}(0)|_{\gamma_0}\leq C(x_0)\|\alpha^{(k,l)}\|^2_{\gamma_0,B_{1,1}}\notag
\end{align}
for $k$, $l$ larger than $k_0(x_0)$ and $l_0(x_0)$.

Moreover, we may assume that the same $k_0(x_0)$ and $l_0(x_0)$ work for all $x$ sufficiently near $x_0$.

By using (\ref{norm localization}), we have that
\begin{align}
|\alpha_{k,l}(x_0)|^2\leq 2 C(x_0)\|\alpha_{k,l}\|^2_{X}\notag
\end{align}
for $k$ and $l$ larger than $k_1(x_0)$ and $l_1(x_0)$  and the same $k_1$, $l_1$ work for all $x$ sufficiently near $x_0$. This proves that $S^{q,k,l}_X(x)$ is dominated by a constant if $X$ is compact.

By (\ref{lemma 2.1}) and the fact that $X$ has finite volume, we see that the sequence $B^{q,k,l}_X(x)$ is  dominated by a $L^1$ function.

 Finally, we have that 
\begin{align}
\limsup_{k,l,\frac{k}{l}\rightarrow \infty}\dim_\mathbb{C} H^{0,q}(X,E^k\otimes F^l)=\limsup_{k,l,\frac{k}{l}\rightarrow \infty}\int_XB^{q,k,l}_X\frac{(\omega_{k,l})^n}{n!}\notag
\end{align}
and then, Fatou's lemma shows that
\begin{align*}
\int_X\limsup_{k,l,\frac{k}{l}\rightarrow \infty}&B^{q,k,l}_X\frac{(\omega_{k,l})^n}{n!}
\leq  \int_X 1_{X(q)}(x)\big|det_{\omega_0}(\Theta)_x\big|\frac{(\omega_{k,l})^n}{n!}\\
&=\frac{k^rl^{n-r}}{r!(n-r)!}\int_{X(q)}(-1)^q\left(\frac{i}{2\pi}c(E)\right)^r\wedge \left(\frac{i}{2\pi}c(F)\right)^{n-r}.
\end{align*}
\end{proof}

\section{The Strong Version of Demailly-Bouche's Holomorphic Morse Inequalities}

Let $\mathcal{H}^q_{\leq \mu_{k,l}}(X,E^k\otimes F^l)$ denote space of the linear span of  the eigenforms of $\Delta_{\overline{\partial}}$ whose eigenvalues are bounded by $\mu_{k,l}$ and $B^{q,k,l}_{\leq \mu_{k,l}}$ the Bergman kernel function of the space $\mathcal{H}^q_{\leq \mu_{k,l}}(X,E^k\otimes F^l)$. In this section, we will prove the following asymptotic equality
\begin{align}
\lim_{k,l,\frac{k}{l}\rightarrow \infty}B^{q,k,l}_{\leq \mu_{k,l}}(x)=1_{X(q)}(x)\big|det_{\omega_0}(\Theta)_x\big|\notag
\end{align}
where $\mu_{k,l}$ is a properly chosen sequence.
\begin{theorem}\label{proposition 5.1}
Assume that $\mu_{k,l}\rightarrow 0$, then the following estimate holds:
\begin{align}
\lim_{k,l,\frac{k}{l}\rightarrow \infty}B^{q,k,l}_{\leq \mu_{k,l}}(x)\leq1_{X(q)}(x)\big|det_{\omega_0}(\Theta)_x\big|.\notag
\end{align}

\end{theorem}
\begin{proof}
The proof is a simple modification of the previous proof of the local Demailly-Bouche's weak holomorphic morse inequalities and in what follows these modifications will be presented. 

The difference is that $\alpha_{k,l}\in \mathcal{H}^q_{\leq \mu_{k,l}}(X,E^k\otimes F^l)$  and we have to prove that all  terms of the form $(\Delta_{\overline{\partial}}^{(k,l)})^m(\alpha^{(k,l)})=(\Delta_{\overline{\partial}}^{(k,l)})^m\beta^{(k,l)}$ vanish in the limit.

 For any ball $B$,
\begin{align}
\|(\Delta_{\overline{\partial}}^{(k,l)})^m\beta^{(k,l)}\|^2_{\gamma_0,B}\leq \|(\Delta_{\overline{\partial}}^{(k,l)})^m\alpha^{(k,l)}\|^2_{\gamma_0, B_{|z'|<r_{k,l},|z''|<r_{k,l}}}\leq \|\Delta^m_{\overline{\partial}}\alpha_{k,l}\|^2_X\notag
\end{align}
and the last term is just a sequence tending to zero because
\begin{align}
\mu_{k,l}^m\rightarrow 0.\notag
\end{align}

Since by assumption, $\alpha_{k,l}$ is of unit norm and in $\mathcal{H}^q_{\leq \mu_{k,l}}(X,E^k\otimes F^l)$, and $\mu_{k,l}\rightarrow 0$, G{\aa}rding's inequality as in Lemma \ref{lemma 3.1} gives that
\begin{align}
\|\beta^{(k,l)}\|^2_{\gamma_0,B,2m}\sim C(\|\beta^{(k,l)}\|^2_{\gamma_0,B}+\|(\Delta_{\overline{\partial}}^{(k,l)})^m\beta^{(k,l)}\|^2_{\gamma_0,B})\leq (C+\mu_{k,l}^{2m})\leq C'\notag
\end{align}
which shows that the conclusion of Lemma \ref{lemma 3.1} is still valid.  Finally, $\Delta_{\overline{\partial},\gamma_0}\beta=0$ as before and the rest of the argument goes through word by word.
\end{proof}

The next lemma provides the sequence that takes the right value at a given point $x\in X(q)$, with "small" Laplacian, that was referred to at the beginning of the section.

\begin{lemma}\label{part2}
Let $c_{\Theta}(x)=1_{X(q)}(x)\big|det_{\omega_0}(\Theta)_x\big|$. For any point $x_0\in X(q)$ there is a sequence $\alpha_{k,l}$ such that $\alpha_{k,l}$ is in $\Omega^q(X,E^k\otimes F^l)$ with
\begin{align}
  &(i) ~~~ |\alpha_{k,l}(x_0)|^2=c_\Theta(x_0)\notag\\
   &(ii) ~~~~ \lim \|\alpha_{k,l}\|^2=1\notag\\
  &(iii) ~~~~  \|(\Delta_{k,l})^m\alpha_{k,l}\|\rightarrow 0\notag
\end{align}
Moreover, there is a sequence $\delta_{k,l}$ independent of $x_0$ and tending to zero, such that
\begin{align}
&(iv)~~~~~\langle\Delta_{k,l}\alpha_{k,l},\alpha_{k,l} \rangle_X\leq \delta_{k,l}.\notag
\end{align}
\end{lemma}
\begin{proof}
We may assume that the first $q$ eigenvalues $\lambda_{x_0,i}$ are  negative, while the remaining eigenvalues are positive.

 Define the following form in $\mathbb{C}^n$:
\begin{align}
\beta(w)=(\frac{|\lambda_1|\cdots|\lambda_n|}{\pi^n})^{\frac{1}{2}}e^{\sum_{i=1}^q\lambda_i|w_i|^2}d\overline{w_1}\wedge\cdots\wedge d\overline{w_q}\notag
\end{align}
so that $|\beta|^2_{\gamma_0}=\frac{|\lambda_1|\cdots |\lambda_n|}{\pi^n}e^{{-\sum_{i=1}^n}|\lambda_i||w_i|^2}$ and $\|\beta\|^2_{\gamma_0,\mathbb{C}^n}=1$. Observe that $\beta$ is in $L^{2,m}_{\gamma_0}$, the Sobolev space with $m$ derivatives in $L^2_{\gamma_0}$, for all  $m$.

Now define $\alpha_{k,l}$ on $X$ by
\begin{align}
\alpha_{k,l}(z):=\chi_{k,l}(\sqrt{k}z',\sqrt{l}z'')\beta(\sqrt{k}z',\sqrt{l}z'')\notag
\end{align}
where $\chi_{k,l}=\chi(\frac{w'}{\sqrt{k}},\frac{w''}{\sqrt{l}})$ and $\chi $ is  a smooth function supported on the unit ball, which equals one on the ball of radius $\frac{1}{2}$.

It is a direct computation that  $|\alpha_{k,l}(x_0)|^2=c_{\Theta}(x_0)$, which implies $(i)$.

To see $(ii)$,  note that
\begin{align}\label{norm1}
\|\alpha_{k,l}\|_X^2=\|\chi_{k,l}\beta\|^2_{\gamma_0,\mathbb{C}^n}=\|\beta\|^2_{\gamma_0,\frac{1}{2}r_{k,l},\frac{1}{2}r_{k,l}}+\|\chi_{k,l}\beta\|^2_{\gamma_0,\geq\frac{1}{2}r_{k,l},\geq\frac{1}{2}r_{k,l} }
\end{align}
and the 'tail' $\|\chi_{k,l}\beta\|^2_{\gamma_0,\geq\frac{1}{2}r_{k,l},\geq\frac{1}{2}r_{k,l}} $ tends to zero, since $\beta$ is in $L^2_{\gamma_0,C^n}$ and $r_{k,l}$ tends to infinity.

Now we show $(iii)$. Changing variables and using (\ref{Laplacian expansion}) gives
\begin{align}
\|(\Delta_{k,l})^m\alpha_{k,l}\|^2_X&\sim \|(\Delta_{\overline{\partial}}^{(k,l)})^m\chi_{k,l}\beta\|^2_{\gamma_0,r_{k,l},r_{k,l}}\notag\\
&= \|(\Delta_{\overline{\partial}}^{(k,l)})^{m-1}(\Delta_{\overline{\partial},\gamma_0}+\epsilon_{k,l}\mathcal{D}_{k,l})\chi_{k,l}\beta\|^2_{\gamma_0,r_{k,l},r_{k,l}}\notag
\end{align}
where $\mathcal{D}_{k,l}$ is a second order partial differential operator, whose coefficients have derivatives that are uniformly bounded in $k$.

To see that this tends to zero first observe that
\begin{align}\label{part1}
\|(\Delta_{\overline{\partial}}^{(k,l)})^{m-1}\Delta_{\overline{\partial},\gamma_0}\chi_{k,l}\beta\|^2_{\gamma_0,r_{k,l},r_{k,l}}
\end{align}
tends to zero. Indeed $\beta$ has been chosen so that $\Delta_{\overline{\partial},\gamma_0}\beta=0$. Moreover $\Delta_{\overline{\partial},\gamma_0}$ is the square of the first order operator  $\overline{\partial}+\overline{\partial}^{*,\gamma_0}$ which also annihilates $\beta$ and obeys a Lebniz like rule, showing that
\begin{align}
\Delta_{\overline{\partial},\gamma_0}\chi_{k,l}\beta=\gamma_{k,l}\beta\notag
\end{align}
where $\gamma_{k,l}$ is a function, uniformly bounded in $k,l$ and supported outside the ball $B_{\frac{1}{2}r_{k,l},\frac{1}{2}r_{k,l}}$ ($\gamma_{k,l}$ contains second derivatives of $\chi_{k,l}$).

Now using (\ref{Laplacian expansion}) again we see that (\ref{part1}) is bounded by the norm of  $\gamma_{k,l}p(w,\overline{w})\beta$, where $p$ is a polynomial, and thus tends to zero estimated by the 'tail' of a  convergence integral, as in (\ref{norm1})-the polynomial does not affect the convergence.

To finish the proof of (iii), it is now enough to show that
\begin{align}
\|(\Delta_{\overline{\partial}}^{(k,l)})^{m-1}\mathcal{D}_{k,l}(\chi_{k,l}\beta)\|^2_{\gamma_0,r_{k,l},r_{k,l}}\notag
\end{align}
is uniformly bounded. As above one sees that the integrand is bounded by the norm of $q(w,\overline{w})\beta$, for some polynomial $q$, which is finite as above.

To prove (iv),  observe that as above,
\begin{align}
\langle \Delta_{k,l}\alpha_{k,l},\alpha_{k,l}\rangle=\|(\overline{\partial}+\overline{\partial}^{*})\alpha_{k,l}\|^2_X\sim \|(\overline{\partial}+\overline{\partial}^{*,(k,l)})(\chi_{k,l}\beta)\|^2_{r_{k,l},r_{k,l}}.\notag
\end{align}

Hence, by Leibniz' rule
\begin{align}
\langle\Delta_{k,l}\alpha_{k,l},\alpha_{k,l}\rangle\sim \|\chi_{k,l}(\overline{\partial}+\overline{\partial}^{*,(k,l)})\beta\|^2_{r_{k,l},{r_{k,l}}}+\frac{C}{r^2_{k,l}}\|\beta\|^2_{C^n}.\notag
\end{align}
Clearly, there is an expansion for the first order operator $(\overline{\partial}+\overline{\partial}^{*})^{(k,l)}$ as in (\ref{Laplacian expansion}), giving
\begin{align}
\langle\Delta_{k,l}\alpha_{k,l},\alpha_{k,l}\rangle\sim \varepsilon_{k,l}(\|\beta\|^2+\sum_{i=1}^{2n}\|\partial_i\beta\|^2)+\frac{C}{r^2_{k,l}}\|\beta\|^2.\notag
\end{align}
Note that even if $\|\beta\|^2$ is independent of the eigenvalues $\lambda_{i,x_0}$, the norms $\|\partial_i\beta\|^2$ do depend on the eigenvalues, and hence on the point $x_0$. But the dependence amounts to a factor of eigenvalues and since $X$ is compact, we deduce that $\|\partial_i\beta\|^2$ is bounded by a constant independent of the point $x_0$. This shows that $\langle\Delta_{k,l}\alpha_{k,l},\alpha_{k,l}\rangle\leq \delta_{k,l}$. Note that $\varepsilon_{k,l}$ also can be taken to be independent of the point $x_0$, by a similar argument. This completes the proof of (iv).

\end{proof}

\begin{theorem}\label{pro 5.3}
Assume that the sequence $\mu_{k,l}$ is such that $\mu_{k,l}\neq 0$ and $\frac{\delta_{k,l}}{\mu_{k,l}}\rightarrow 0$, where $\delta_{k,l}$ is the sequence appearing in Lemma \ref{part2}. Then for any point $x\in X(q)$, the following holds
\begin{align}
\liminf_{k,l,\frac{k}{l}\rightarrow \infty} B^{q,k,l}_{\leq \mu_{k,l}}(x)\geq1_{X(q)}(x)\big|det_{\omega_0}(\Theta)_x\big|.\notag
\end{align}
\end{theorem}

\begin{proof}
Let $\{\alpha_{k,l}\}$ be the sequence that Lemma \ref{part2} provides and decompose it with respect to the orthongonal decomposition $\Omega^{0,q}(X,L^k)=\mathcal{H}^q_{\leq \mu_{k,l}}(X,E^k\otimes F^l )\oplus \mathcal{H}^q_{> \mu_{k,l}}(X,E^k\otimes F^l )$ induced by the spectral decomposition of the elliptic operator $\Delta_{\overline{\partial}}$:
\begin{align}
\alpha_{k,l}=\alpha_{1,k,l}+\alpha_{2,k,l}.\notag
\end{align}

Firstly, we prove that
\begin{align}\label{5.5}
\lim |\alpha^{(k,l)}_2(0)|^2=0.
\end{align}

As in the proof of Lemma \ref{lemma 3.1}, we have that
\begin{align}
|\alpha_2^{(k,l)}(0)|^2\leq C(x)(\|\alpha^{(k,l)}_2\|^2_{B_1}+\|(\Delta_{\overline{\partial}}^{(k,l)})^m\alpha_2^{(k,l)}\|^2_{B_1}).\notag
\end{align}

To see that the first term tends to zero, observe that by the spectral decomposition of $\Delta_{k,l}$:
\begin{align}
\|\alpha_2^{(k,l)}\|^2_X\leq \frac{1}{\mu_{k,l}}\langle \Delta_{k,l}\alpha_{2,k,l},\alpha_{2,k,l}\rangle\leq \frac{\delta_{k,l}}{\mu_{k,l}}.\notag
\end{align}
Furthermore, the second term also tends to zero:
\begin{align}
\|(\Delta_{\overline{\partial}}^{(k,l)})^m\alpha_2^{(k,l)}\|^2_{B_1}\leq \|(\Delta_{\overline{\partial},\gamma_0})^m\alpha_{2,k,l}\|^2_X\rightarrow 0\notag
\end{align}
by $(iii)$ in Lemma \ref{part2}. Finally,
\begin{align}
S^{q,k,l}_{\leq\mu_{k,l}}&\geq \frac{|\alpha_{1,k,l}(0)|^2}{\|\alpha_{1,k,l}\|^2_X}\notag\\
&\geq |\alpha_{1,k,l}(0)|^2\notag\\
&=|\alpha_{k,l}(0)-\alpha_{2,k,l}(0)|^2.\notag
\end{align}

By (\ref{5.5}), we can see that this tends to the limit of $|\alpha_{k,l}(0)|^2$, which proves the theorem according to (\ref{lemma 2.1}) and $(i)$ in Lemma \ref{part2}. 
\end{proof}

Now we can prove the following asymptotic equality:
\begin{theorem}[=Theorem \ref{Main theorem 2}]
Let $(X,\omega)$ be a compact hermitian manifold. Then
\begin{align}
\lim_{k,l,\frac{k}{l}\rightarrow \infty}B^{q,k,l}_{\leq \mu_{k,l}}(x)=1_{X(q)}(x)\big|det_{\omega_0}(\Theta)_x\big|\notag
\end{align}
for some sequence $\mu_{k,l}$ tending to zero.
\end{theorem}
\begin{proof}
Let $\mu_{k,l}=\sqrt{\delta_{k,l}}$. The theorem then follows immediately from Theorem \ref{proposition 5.1} and Theorem \ref{pro 5.3} if $x\in X(q)$. If $x$ is outside of $X(q)$, then the upper bound given by Theorem \ref{proposition 5.1} shows that $\lim_{k,l}B^{q,k,l}_{\leq \mu_{k,l}}(x)=0$, which finish the proof of the theorem.
\end{proof}

\subsection*{Proof of Strong Demailly-Bouche's holomorphic Morse inequalities} By combination of the following lemma, we can get the strong version (\ref{Bou2}) in Theorem \ref{Bou}.
\begin{lemma}[c.f.\cite{Dem89}]\label{Wit}
Denote by  $h^q$ the complex dimension of $ H^q(X,E^k\otimes F^l)$ and by $h^q_{\leq \mu_{k,l}}$ the complex dimension of $H^q_{\leq \mu_{k,l}}(X,E^k\otimes F^l)$ respectively, then the following holds:
\begin{align}
\sum_{0\leq j\leq q}(-1)^{q-j}h^j\leq \sum_{0\leq j\leq q}(-1)^{q-j}h^j_{\leq \mu_{k,l}}.\notag
\end{align}
\end{lemma}

\begin{theorem}
In fact, all the above inequalities we proved is for the case of $E^k\otimes F^l$ but they generalizes to the case of  $E^k\otimes F^l\otimes G$ straightly, where $G$ is a holomorphic vector bundle with rank $g\geq 2$. Since the estimates for the extremal functions $S^{q,k,l}_X$ is the same, while the estimates for $B^{q,k,l}_X$ are modified by a factor $g$ in the right hand side.
\end{theorem}

\begin{remark}Let $E$ is a holomorphic line bundle, the kernel of the curvature of $E$ is a tangent space of a foliation of codimension $r$. Let $G$ be a holomorphic vector bundle on $X$. Then it is proved [{c.f. \cite[Corollary 0.2]{Bouche 91}} that  there exists a constant $C$ such that for $q=0,1, \cdots, n$, and $k\rightarrow \infty$,
\begin{align}
\dim H^q(X, E^k\otimes G)\leq Ck^r.\notag
\end{align}
\end{remark}
\begin{example}
 Let $M$ be a compact complex manifold. $E\rightarrow M$ be the trivial holomorphic vector bundle $\mathbb{C}^r\times M$ with $r\geq 2$. Equipped $E$ with the trivial hermitian metric, which can induce a hermitian metric $h$ on the holomorphic line bundle $O_E(1)$, which have positive curvature along the fiber and vanishes along the horizontal direction, i.e. the curvature $c(E)$ of $E$ has rank $r$ and $M$ is a foliation of the kernel of $c(E)$. Then from the above Remark, we can get that $\dim_\mathbb{C}H^q(P(E^*),O_E(k))\leq Ck^{r-1}$.
\end{example}

\begin{remark}Thanks to our local version of Demailly-Bouche's  holomorphic Morse inequalities, combined with the technique developed in \cite{Berman 06}, it is possible to extend the corresponding result in \cite{Berman 06}  of asymptotics of  super Toeplitz operator and  sampling sequences to the case of $E^k\otimes F^l$, where $E$ and $F$ are holomorphic line bundles satisfying the assumptions in Theorem \ref{Bou}. Since our main purpose of this paper is to develop local version of Demailly-Bouche's holomorphic Morse inequalities, we do not investigate such applications in this paper.
\end{remark}

\begin{remark}In \cite{HM12}, Hsiao-Li proved Morse type inequalities on CR manifolds with transversal CR $S^1$-action. The main strategy of their work is to adapt localization procedure in the case of  complex manifolds to the case of CR manifolds with transversal CR $S^1$-action. In view of this, it is natural to ask the following 
	\begin{question}
Let $X$ be a compact connected CR manifold with transversal CR $S^1$-action, let $\cL$ be the Levi form associated to $X$. Suppose that $\cL$ is everywhere degenerate and the maximal rank of $\cL$ is $r$ and the kernel of $\cL$ is foliated, i.e. there is a foliation $Y$ of $X$ of complex codimension $r$, such that the tangent space of the leaf at each point $x\in X$ is contained in the kernel of $\cL$. Let $L$ be a rigid Hermitian CR line bundle over $X$. Then can we get a better estimate  of the dimension of the Fourier components $H^q_{b,k}(X,L^l)$ of the Kohn-Rossi cohomology (say as in Theorem \ref{Bou}) compared with Hsiao-Li's Morse type inequalities for rigid Hermitian CR line bundles, where $k,l,\frac{k}{l}\rightarrow +\infty$?
		\end{question}

	\end{remark}

\subsection*{Acknowledgement} The author is deeply grateful to Professor Xiangyu Zhou for his constant help, encouragement, and inspiring discussions from which the author benifits a lot.


\bibliographystyle{fancy}


\end{document}